\documentclass[12pt]{amsart}
\usepackage{SSdefn}
\usepackage{eucal, mathtools}
\usepackage{tikz}
\usepackage{tikz-cd}

\newcommand{\FB}{\mathbf{FB}}

\newcommand{\FI}{\mathbf{FI}}

\newcommand{\VI}{\mathbf{VI}}

\newcommand{\VB}{\mathbf{VB}}
\newcommand{\GrVec}{\mathrm{GrVec}}
\newcommand{\cor}{\mathrm{cor}}

\DeclarePairedDelimiter\abs{\lvert}{\rvert}

\title[Periodicity in the cohomology of finite general linear groups]{Periodicity in the cohomology of finite general\\ linear groups via $q$-divided powers}

\date{November 14, 2020}

\author{Rohit Nagpal}
\address{Department of Mathematics, University of Michigan, Ann Arbor, MI}
\email{\href{mailto:rohitna@umich.edu}{rohitna@umich.edu}}
\urladdr{\url{http://www-personal.umich.edu/~rohitna/}}

\author{Steven V Sam}
\address{Department of Mathematics University of California, San Diego, CA}
\email{\href{mailto:ssam@ucsd.edu}{ssam@ucsd.edu}}
\urladdr{\url{http://math.ucsd.edu/~ssam/}}

\author{Andrew Snowden}
\address{Department of Mathematics, University of Michigan, Ann Arbor, MI}
\email{\href{mailto:asnowden@umich.edu}{asnowden@umich.edu}}
\urladdr{\url{http://www-personal.umich.edu/~asnowden/}}

\thanks{RN was partially supported by NSF DMS-1638352. SS was partially supported by NSF DMS-1849173 and a Sloan Fellowship. AS was supported by NSF DMS-1453893.}

\begin{document}

\begin{abstract}
We show that $\bigoplus_{n \ge 0} \rH^t(\GL_{n}(\bF_q), \bF_{\ell})$ canonically admits the structure of a module over the $q$-divided power algebra (assuming $q$ is invertible in $\bF_{\ell}$), and that, as such, it is free and (for $q \neq 2$) generated in degrees $\le t$. As a corollary, we show that the cohomology of a finitely generated $\VI$-module in non-describing characteristic is eventually periodic in $n$. We apply this to obtain a new result on the cohomology of unipotent Specht modules. 
\end{abstract}

\maketitle

\section{Introduction}

\subsection{Background}

Fix a field $\bk$ of positive characteristic $\ell$, and let $\bE^i_n = \rH^i(\fS_n)$ be the $i$th cohomology of the symmetric group $\fS_n$ with coefficients in $\bk$. The space $\bE=\bigoplus_{i,n} \bE^i_n$ admits the structure of a bi-graded algebra: the multiplication map (which is called the {\bf transfer product}) is the composite
\begin{displaymath}
\rH^i(\fS_n) \times \rH^j(\fS_m) \to \rH^{i+j}(\fS_n \times \fS_m) \to \rH^{i+j}(\fS_{n+m}),
\end{displaymath}
where the first map comes from the K\"unneth isomorphism and the second is corestriction. Let $d \colon \bE^i_n \to \bE^i_{n-1}$ be the restriction map along the inclusion $\fS_{n-1} \subset \fS_n$. We recall two well-known results:
\begin{itemize}
\item A theorem of Nakaoka \cite{nakaoka} asserts that $d$ is an isomorphism for $n>2i$.
\item A theorem of Dold \cite[Lemma~1,2, Theorem~1]{dold} asserts that $d$ is a derivation of $\bE$.
\end{itemize}
Let $\bE^i=\bigoplus_{n \ge 0} \bE^i_n$, and let $\bD=\bE^0$. Then $\bD$ is a graded algebra, which is none other than the divided power algebra over $\bk$ in a single variable, and each $\bE^i$ is a $\bD$-module.

In \cite{periodicity}, the first and third authors reinterpreted the above two results as follows: Dold's theorem implies that $d$ defines a connection on $\bE^i$ (relative to the derivation $d$ on $\bD$); and Nakaoka's theorem implies that the kernel of this connection is supported in degrees $\le 2i$. As a consequence, we find that $\bE^i$ is a free $\bD$-module generated in degrees $\le 2i$. We applied this point of view to study the cohomology of $\FI$-modules. We showed that if $M$ is a finitely generated $\FI$-module and $i \ge 0$ then $\bigoplus_{n \ge 0} \rH^i(\fS_n, M_n)$ is canonically a $\bD$-module, and is ``nearly'' finitely presented, in the sense that it agrees with a finitely presented $\bD$-module outside of finitely many degrees. As a corollary, we reproved the main result of \cite{nagpal} that $n \mapsto \dim \rH^i(\fS_n, M_n)$ is eventually periodic in $n$.

\subsection{Main results}

The purpose of this paper is to generalize the above picture to finite general linear groups. The analog of Nakaoka's theorem in this setting follows from classical work of Quillen \cite{quillen, quillen-notebook}. Our main theorem is the analog of Dold's theorem, which appears to be new. Once this result is established, the methods of \cite{periodicity} give us periodicity results for the cohomology of $\VI$-modules.

We now state our results precisely. Let $q$ be a prime power that is invertible in $\bk$, and let $\bG_n=\GL_n(\bF_q)$ be the finite general linear group over the field of $q$ elements. Let $\bE_n^i=\rH^i(\bG_n)$ be the cohomology of $\bG_n$ with coefficients in $\bk$. Once again, this is an algebra: the multiplication map is the composite
\begin{displaymath}
\rH^i(\bG_n) \times \rH^j(\bG_m) \to \rH^{i+j}(\bG_n \times \bG_m) \to \rH^{i+j}(\bP_{n,m}) \to \rH^{i+j}(\bG_{n+m}), 
\end{displaymath}
where the first map comes from the K\"unneth isomorphism, the second map is restriction to the parabolic $\bP_{n,m} \subset \bG_{n+m}$ with Levi factor $\bG_n \times \bG_m$, and the third map is corestriction. This operation is poorly understood, but can be useful: for example, \cite{lahtinen-sprehn} have used it to construct nontrivial cohomology classes in the case when the coefficient field has characteristic $p$. We let $d \colon \bE_n^i \to \bE_{n-1}^i$ be the composite
\begin{displaymath}
\rH^i(\bG_n) \to \rH^i(\bP_{n-1,\ol{1}}) \to \rH^i(\bG_{n-1})
\end{displaymath}
where $\bP_{n-1,\ol{1}}$ denote the subgroup of $\bP_{n-1,1}$ consisting of matrices whose last diagonal entry is 1, the first map is restriction, and the second is the canonical isomorphism (as $q$ is invertible in $\bk$, the inflation map is an isomorphism). Our main theorem is then:

\begin{theorem}
\label{thm:q-construction}
The map $d$ is a (right) $q$-derivation on $\bE$, that is, for $x \in \bE_n$ and $y \in \bE_m$ we have
\begin{displaymath}
d(xy)=xd(y) + q^m d(x) y. 
\end{displaymath} Moreover, $d$ is surjective.
\end{theorem}

\begin{remark}
  The fact that $d$ is surjective can be deduced from \cite{quillen}, namely using \cite[Theorem 4]{quillen} which describes a basis for $\rH^i(\bG_n)$ using characteristic classes of the standard representation, and \cite[Proposition 3]{quillen} which describes how these characteristic classes behave under the restriction $\bG_n \to \bG_{n-1}$.
\end{remark}

Let $\bD=\bE^0$, which is a subalgebra of $\bE$. It is isomorphic to the $q$-divided power algebra over $\bk$ in one variable (see \S \ref{s:qdp}). This isomorphism can be obtained using the method in \cite[Proposition~11]{lahtinen-sprehn}. Alternatively, a more formal proof in the $q=1$ case is given in \cite[Proposition~4.3]{periodicity}, which easily extends to the general case. Under this isomorphism, $d$ corresponds to the derivation of $\bD$ given by $x^{[n]} \mapsto x^{[n-1]}$.

The space $\bE^i$ is naturally a $\bD$-module, and $d$ induces a connection on it. We obtain the following corollary by combining the theorem above and the homological stability for finite general linear groups. 

\begin{corollary}	\label{cor:free-D-module}
The space $\bE^i$ is a graded free $\bD$-module generated in degrees $\le 2i$. In fact, if $q \neq 2$ then it is generated in degrees $\le i$.
\end{corollary}

\subsection{Applications}

Corollary~\ref{cor:free-D-module} shows that the cohomology of $\bG_n$ with coefficients in the trivial module admits a nice piece of additional structure. It is natural to look for a generalization of this result to non-trivial coefficients. Of course, one can only hope for such a result if the coefficient systems vary in a reasonable way. Let $\VI$ be the category of finite dimensional $\bF_q$-vector spaces and linear injections. A finitely generated $\VI$-module $M$ gives rise to representations $M_n=M(\bF_q^n)$ of $\bG_n$ for all $n$, and provides a good source of ``reasonable'' coefficient systems. We prove the following result concerning the cohomology of these modules (see Theorem~\ref{thm:VImod} for a quantitative statement):

\begin{theorem}
Let $M$ be a finitely generated $\VI$-module and let $t \ge 0$. Then $\bigoplus_{n \ge 0} \rH^t(\bG_n, M_n)$ is isomorphic to a finitely presented $\bD$-module, outside of finitely many degrees. In particular, $n \mapsto \rH^t(\bG_n, M_n)$ is eventually periodic in $n$, and the period divides a number of the form $w \ell^i$, where $w$ is the order of $q$ in $\bk^*$ and $\ell$ is the characteristic of $\bk$.
\end{theorem}

 Denote, by $\bM_{\mu}$, the unipotent Specht module of $\bG_d$ corresponding to a partition $\mu$ of $d$. We use the theorem above to obtain the following result:

\begin{theorem}
	\label{thm:unipotent} 
	Assume that $\ell$ and $q$ are coprime.  Let $\mu$ be a partition of $d$, and let $s \ge 0$ be the smallest integer such that $w \ell^{s} \ge 2t +7d$. Then $n \mapsto \dim_{\bk} \rH^t(\bG_n, \bM_{\mu[n]})$ is periodic in the range $n \ge \max(d + 2t, 4d + 3)$ with period dividing $w \ell^{s + 2d+1}$. Additionally, if $q \neq 2$, then $s$ may be taken to be the smallest integer such that $w \ell^{s} \ge t +4d$.
\end{theorem}

\subsection{Overview of the proof}

We give two proofs of Theorem~\ref{thm:q-construction}. The second one is a straightforward, if somewhat involved, computation in cohomology. This proof is close to Dold's proof for the symmetric group. In particular, Lemma~\ref{lem:mid-portion} is the analogue of \cite[Lemma 1]{dold}. The first proof, on the other hand, derives the result from general considerations. It is conceptually more clear, and applies in some other situations. We briefly summarize the idea here.

Let $\cA$ be a $\bk$-linear abelian category equipped with a braided tensor product and an object $\bk\{1\}$, regarded as a shift of the unit object. Let $R$ be an algebra in $\cA$. We define a derivation of degree $-1$ on $R$ to be a map $d \colon R \to  R \otimes \bk\{1\}$ satisfying the Leibniz rule. We note that to make sense of the Leibniz rule, one has to commute $\bk\{1\}$ and $R$ at some point, which is why we require a braiding. Fix such a $d$ on $R$. Suppose now that $\cB$ is a second $\bk$-linear abelian category equipped with a braided tensor product, and that $\Gamma \colon \cA \to \cB$ is a left-exact strict braided monoidal functor. Assuming $\Gamma$ satisfies some mild technical conditions, its right-derived functor $\rR \Gamma$ is also strict braided monoidal, and so $(\rR \Gamma)(R)$ is an algebra in $\rD(\cB)$ and $d$ induces on it a derivation of degree $-1$. 

Now, let $\FB$ be the groupoid of finite sets, and let $\Mod_{\FB}$ be the category of functors $\FB \to \Mod_{\bk}$. Then $\Mod_{\FB}$ is a symmetric monoidal category under Day convolution (with respect to disjoint union). Let $\bA$ be the $\FB$-module defined by $\bA(S)=\bk$ for all $S$, which is naturally an algebra. We show that $\bA$ admits a natural degree $-1$ derivation $d$. The functor
\begin{displaymath}
\Gamma \colon \Mod_{\FB} \to \GrVec, \qquad
\Gamma(M) = \bigoplus_{n \ge 0} M([n])^{\fS_n}
\end{displaymath}
is naturally a strict symmetric monoidal functor. We thus see that $d$ induces a derivation on the algebra $\rR \Gamma(\bA)$. This recovers Dold's theorem.

Our first proof of Theorem~\ref{thm:q-construction} is entirely parallel to this. Let $\VB$ be the groupoid of finite dimensional $\bF_q$-vector spaces. The category $\Mod_{\VB}$ admits a natural braiding, due to Joyal--Street. Let $\bA$ be the $\VB$-module defined by $\bA(V)=\bk$ for all $V$. Then $\bA$ is naturally an algebra, and we show that it admits a natural degree $-1$ derivation $d$. We define $\Gamma$ on $\Mod_{\VB}$ similarly to the above, but using $\bG_n$-invariants. We show that it is strict braided monoidal (with respect to a non-standard braiding on $\GrVec$). Thus $d$ induces a derivation on $\rR \Gamma(\bA)$, which yields the main theorem.

\subsection{Outline}

In \S \ref{s:congen} we explain the general formalism of derivations in braided monoidal categories, and how this often leads to derivations on cohomology. In \S \ref{s:symmetric}, we explain in some more detail how this formalism can be used to recover Dold's theorem. In \S \ref{s:qdp}, we review the $q$-divided power algebra, and some other basic $q$-analogs. In \S \ref{s:gl}, which is the heart of the paper, we give our two proofs of Theorem~\ref{thm:q-construction}. Finally,  in \S \ref{s:vi}, we give the applications to $\VI$-modules and to the cohomology of unipotent Specht modules. 

\subsection{Notation}

We list some of the important notation used throughout the paper:
\begin{itemize}
\item $\bk$: the coefficient field, typically of finite characteristic $\ell$.
\item $q$: a prime power that is invertible in $\bk$.
\item $w$: the order of $q$ in the multiplicative group $\bk^*$.
\item $\bG_n$: the finite general linear group $\GL_n(\bF_q)$.
\item $\fS_n$: the symmetric group on $n$ letters.
\item $\FB$: the groupoid of finite sets.
\item $\VB$: the groupoid of finite dimensional $\bF_q$-vector spaces.
\item $\Vec$: the category of $\bk$-vector spaces.
\item $\GrVec$: the category of graded $\bk$-vector spaces.
\item $\Mod_{\cC}$, where $\cC$ is a category: the category of functors $\cC \to \Vec$.
\end{itemize}

\subsection*{Acknowledgements}

We thank Andrew Putman for helpful conversations.

\section{Derivations on cohomology}
\label{s:congen}

Fix a field $\bk$. Let $\cA$ be a $\bk$-linear category equipped with a bilinear braided monoidal operation $\otimes$. We suppose given an object $\bk\{1\}$ of $\cA$, which we think of as a ``shift'' of the unit object.

Let $R$ be an algebra in $\cA$, assumed to be unital and associative. A {\bf derivation (of degree $-1$)} on $R$ is a map $d \colon R \to R \otimes \bk\{1\}$ satisfying the Leibniz rule, in the following sense. Let $m \colon R \otimes R \to R$ be the multiplication map, and let $\beta$ be the braiding. Consider the following three compositions: \begin{align*}
a &\colon R \otimes R \stackrel{d \otimes 1}{\longrightarrow} R \otimes \bk\{1\} \otimes R \stackrel{1 \otimes \beta}{\longrightarrow} R \otimes R \otimes \bk\{1\}   \stackrel{m \otimes 1}{\longrightarrow} R \otimes \bk\{1\} \\
b &\colon R \otimes R \stackrel{1 \otimes d}{\longrightarrow} R \otimes R \otimes \bk\{1\}   \stackrel{m \otimes 1}{\longrightarrow} R \otimes \bk\{1\} \\
c &\colon R \otimes R \stackrel{m}{\longrightarrow} R \stackrel{d}{\longrightarrow} R \otimes \bk\{1\}.
\end{align*} 
The Leibniz rule then takes the form $c=a+b$.

Let $R$ be an algebra in $\cA$ equipped with a derivation $d$, and let $M$ be a (right) $R$-module. A {\bf connection} on $M$ (of degree $-1$) is a map $\nabla \colon M \to M\otimes \bk\{1\}$ satisfying the obvious analog of the Leibniz rule.

Suppose now that $\cB$ is a second $\bk$-linear category equipped with a bilinear braided monoidal operation and an object $\bk\{1\}$, and $\Gamma \colon \cA \to \cB$ is a $\bk$-linear strict braided monoidal functor taking $\bk\{1\}$ to $\bk\{1\}$. Then $\Gamma$ preserves algebras, derivations, and connections. Precisely, if $R$ is an algebra in $\cA$ with a derivation $d$ and $M$ is an $R$-module with a connection $\nabla$ then $\Gamma(R)$ is an algebra in $\cB$ with a derivation $\Gamma(d)$, and $\Gamma(M)$ is a $\Gamma(R)$-module with a connection $\Gamma(\nabla)$.

Now suppose that we are in the following situation: $\cA$ and $\cB$ are abelian with enough injectives, the tensor products are exact, $\Gamma$ is left-exact, $\bk\{1\}$ is $\Gamma$-acyclic, and any tensor product of injectives in $\cA$ is $\Gamma$-acyclic. Then the bounded below derived categories $\rD^+(\cA)$ and $\rD^+(\cB)$ inherit bilinear braided monoidal structures, and the derived functor $\rR \Gamma$ of $\Gamma$ is a $\bk$-linear strict braided monoidal functor taking $\bk\{1\}$ to $\bk\{1\}$ (thought of as objects in the derived category). Thus $\rR \Gamma$ preserves algebras, derivations, and connections. In particular, if $R$ is an algebra in $\cA$ with a derivation $d$, and $M$ is an $R$-module with a connection $\nabla$, then $\rR \Gamma(R)$ is an algebra in $\rD^+(\cB)$ with a derivation, and $\rR \Gamma(M)$ is an $\rR \Gamma(R)$-module with a connection.

Finally, suppose that $S$ is an algebra in $\rD^+(\cB)$ supported in non-negative cohomological degrees equipped with a derivation. Then one easily sees that $\bigoplus_{i \ge 0} \rH^i(S)$ is an algebra in $\cB$ with a derivation. Furthermore, $\rH^0(S)$ is a subalgebra of $\bigoplus_{i \ge 0} \rH^i(S)$ closed under the derivation, and each $\rH^i(S)$ is a module over it with connection. Thus, in the situation of the previous paragraph, each $\rR^i \Gamma(R)$ is a module over $\Gamma(R)$ with a connection.

\begin{remark}
It is not necessary to assume that the tensor products are exact: we only did this to simplify exposition. Also, rather than assuming the monoidal structures are braided, it suffices to assume that $\bk\{1\}$ belongs to the Drinfeld center.
\end{remark}

\section{The symmetric group revisited}
\label{s:symmetric}

In this section, we explain how the methods of \S \ref{s:congen} can be used to recover Dold's theorem. Our proof of Theorem~\ref{thm:q-construction} will closely parallel this discussion.

Let $\FB$ be the groupoid of finite sets and let $\Mod_{\FB}$ be the category of $\FB$-modules, i.e., the functor category $\Fun(\FB, \Vec)$. The category $\Mod_{\FB}$ carries a natural symmetric tensor product, called {\bf Day convolution},  defined by
\begin{displaymath}
(M \otimes N)(S) = \bigoplus_{S = A \amalg B} M(A) \otimes N(B).
\end{displaymath}
Let $\bk\{1\}$ denote the $\FB$-module that is $\bk \cdot e_i$ on one-element sets $\{i\}$ and zero on other sets; the symbol $e_i$ is simply a name for a basis vector.

Let $\bA$ be the $\FB$-module which takes each set to the trivial representation. For notation, we write $\bA(S)=\bk \cdot t^S$; the power of $t$ is just a name for a basis vector. This $\bA$ is naturally an algebra in $\Mod_{\FB}$ with the product map $(\bA \otimes \bA)(S) \to \bA(S)$ given by $t^A \otimes t^B \mapsto t^{S}$ for all decompositions $S = A \amalg B$. We note that the category of right $\bA$-modules is equivalent to the category of $\FI$-modules, where $\FI$ is the category of finite sets and injective functions \cite[Proposition~7.2.5]{catgb}. We have a map of $\FB$-modules
\begin{displaymath}
d \colon \bA \to \bA \otimes \bk\{1\}, \qquad
d(t^S) = \sum_{i \in S} e_i \otimes t^{S \setminus \{i\}}.
\end{displaymath}
One readily verifies that this is a derivation of $\bA$ of degree $-1$.

Let $\GrVec$ denote the category of non-negatively graded vector spaces. Let $\bk\{1\}$ in $\GrVec$ be the object that is $\bk$ in degree~1, and~0 in other degrees. Define a functor
\begin{displaymath}
\Gamma \colon \Mod_{\bA} \to \GrVec, \qquad
\Gamma(M) = \bigoplus_{n \ge 0} M([n])^{\fS_n}.
\end{displaymath}
The $\Gamma$ is a strict monoidal functor \cite[Proposition~4.1]{periodicity}, and takes $\bk\{1\}$ to $\bk\{1\}$. We thus see from the general considerations of \S \ref{s:congen} that $\bD = \bigoplus_{n \ge 0} \rH^0(\fS_n, \bk)$ is an algebra in $\GrVec$ with derivation, and that $\bE^i = \bigoplus_{n \ge 0} \rH^i(\fS_n, \bk)$ is a $\bD$-module with connection. It is not difficult to see that $\bD$ is naturally identified with the divided power algebra, and that the derivation takes $x^{[n]}$ to $x^{[n-1]}$. The connection on $\bE^i$ is given by the composition
\begin{displaymath}
\rH^i(\fS_n, \bk) = \rH^i(\fS_n, \bA_n) \to \rH^i(\fS_n, (\bA \otimes \bk\{1\})_n) \cong \rH^i(\fS_{n-1}, \bk),
\end{displaymath}
which is simply restriction. Note that $(\bA \otimes \bk\{1\})_n \cong \bk^n \cong \Ind_{\fS_{n-1}}^{\fS_n}(\bk)$, and the isomorphism above is simply Shapiro's lemma. We have thus recovered Dold's theorem.

\section{The $q$-divided power algebra}
\label{s:qdp}

\subsection{$q$-binomial coefficients}

Fix a field $\bk$ and an element $q \in \bk^*$. Define
\begin{displaymath}
[n]_q = 1 + q + \cdots + q^{n-1}, \qquad
[n]_q! = [n]_q [n-1]_q \cdots [1]_q
\end{displaymath}
and
\begin{displaymath}
{n \brack m}_q = \frac{[n]_q!}{[m]_q! [n-m]_q!}.
\end{displaymath}
In the third definition, one should regard the numerator and denominator as polynomials in $q$, perform the division as polynomials---which yields a polynomial---and then evaluate at $q$. The quantity ${n \brack m}_q$ is called the {\bf $q$-binomial coefficient} (or {\bf Gaussian binomial coefficient}). When $q$ is a prime power, it counts the number of $m$-dimensional linear subspaces of $\bF_q^n$, and when $q=1$ it reverts to the usual binomial coefficient. These coefficients satisfy the following analogs of Pascal's identity:
\begin{equation} \label{eq:q-pascal}
{n \brack m}_q = q^m {n-1 \brack m}_q + {n-1 \brack m-1}_q = {n-1 \brack m}_q + q^{n-m} {n-1 \brack m-1}_q.
\end{equation}

\subsection{$q$-derivations}

Let $R$ be a non-negatively graded $\bk$-algebra. A (right) {\bf $q$-derivation} of degree $-n$ on  $R$ is a $\bk$-linear map $d \colon R \to R[n]$ that satisfies \[d(a b) = ad(b) + q^{n\deg(b)} d(a) b \] for homogeneous elements $a,b \in R$. Fix such a $d$. A {\bf $q$-connection} (relative to $d$) on a graded right $R$-module $M$ is a $\bk$-linear map $\nabla \colon M \to M[n]$ that satisfies $\nabla(ma) = md(a) + q^{n \deg(a)}  \nabla(m) a$ for every homogeneous $a \in R$ and $m \in M$.  

\begin{proposition} \label{prop:iterated-q}
Let $M$ be a graded $R$-module. Let $d$ be a degree $-1$ $q$-derivation on $R$, and let $\nabla$ be a $q$-connection on $M$. Then we have
\begin{displaymath}
\nabla^n (ma) = \sum_{i=0}^{n} q^{(n-i)(\deg(a) -i)} {n \brack i}_q  \nabla^{n-i}(m) d^i(a).
\end{displaymath}
for all homogeneous $a \in R$ and $m \in M$.
\end{proposition}
      
\begin{proof}
The formula can easily be verified by induction on $n$ and \eqref{eq:q-pascal}.
\end{proof}

\begin{corollary} \label{cor:itder}
Suppose that ${n \brack i}_q=0$ for $0<i<n$. Then $d^n$ is a $q$-derivation on $R$ and $\nabla^n$ is a $q$-connection on $M$.
\end{corollary}

\subsection{The $q$-divided power algebra}

Let $\bD=\bigoplus_{n \ge 0} \bk x^{[n]}$, where $x^{[n]}$ is an element of degree $n$. We define a multiplication on $\bD$ by the formula
\begin{displaymath}
x^{[n]} \cdot x^{[m]} = {n+m \brack n}_q x^{[n+m]}.
\end{displaymath}
It is well-known that this defines on $\bD$ the structure of a commutative, associative, unital algebra. This algebra is called the {\bf $q$-divided power algebra} (in one variable). When $q=1$, it reverts to the usual divided power algebra.

\begin{proposition}
The map $d \colon \bD \to \bD[1]$ given by $d(x^{[n]}) = x^{[n-1]}$ is a $q$-derivation.
\end{proposition}

\begin{proof}
  We have  
  \begin{align*}
  x^{[i]} d(x^{[j]}) + q^j d(x^{[i]}) x^{[j]}  &= {i+j-1 \brack i}_q x^{[i+j-1]} + q^j {i+j-1 \brack i - 1}_q x^{[i+j-1]}\\
  &= {i+j \brack i}_q x^{[i+j-1]}\\
  &=  d(x^{[i]} x^{[j]}). \qedhere
  \end{align*}
\end{proof}

\subsection{Connections on $\bD$-modules}

For a graded $\bD$-module $M$, we set $\ol{M} = M/\bD_+M$, and we denote the image of an element $m \in M$ in $\ol{M}$ by $m_0$. We note that $\bD$ has {\bf Taylor expansions}, in the sense that the map $\bD \to \ol{\bD} \otimes_{\bk} \bD $ given by $f \mapsto \sum_{n \ge 0}  d^n(f)_0 \otimes x^{[n]}$ is an isomorphism.

\begin{proposition}  \label{prop:connection}
  Let $M$ be a graded right $\bD$-module equipped with a $q$-connection $\nabla$. Define a map
  \begin{displaymath}
    \Phi \colon M \to \overline{M} \otimes_{\bk} \bD, \qquad
    \Phi(m) = \sum_{n \ge 0} (\nabla^n m)_0 \otimes x^{[n]}.
  \end{displaymath}
  Then:
  \begin{enumerate}[\rm \indent (a)]
  \item $\Phi$ is an isomorphism of $\bD$-modules.
  \item Under $\Phi$, the connection $\nabla$ corresponds to $1 \otimes d$.
  \item The map $\Phi$ identifies $\ker(\nabla)$ with $\overline{M}$.
  \item The natural map $\ker(\nabla) \otimes_{\bk} \bD     \to M$ is an isomorphism of $\bD$-modules.
  \item $\nabla$ is surjective.
  \end{enumerate}
\end{proposition}

\begin{proof}
  Here we only show that $\Phi$ is $\bD$-linear. The rest of the proof is entirely analogous to the one in \cite[Proposition 3.2]{periodicity}. 
	
  To see the $\bD$-linearity of $\Phi$, let $m \in M$ and $a \in \bD$ be homogeneous elements. By Proposition~\ref{prop:iterated-q}, we have
  \[
    \nabla^n (ma) = \sum_{i=0}^{n} q^{(n-i)(\deg(a) -i)} {n \brack i}_q  \nabla^{n-i}(m) d^i(a).
  \]
  Note that $d^i(a)_0$ is nonzero only if $i = \deg(a)$. Thus we have
	\begin{displaymath}
	\begin{split}
	\Phi(ma)
	&= \sum_{n \ge 0} \sum_{i+j=n} q^{(n-i)(\deg(a) -i)} {n \brack i}_q   \nabla^j(m)_0 d^i(a)_0 \otimes x^{[n]} \\
	&=\left( \sum_{j \ge 0}  \nabla^j(m)_0 \otimes x^{[j]} \right) \left( \sum_{i \ge 0}  d^i(a)_0 x^{[i]} \right)  \\
	&= \Phi(m)a.
	\end{split}
	\end{displaymath}
This proves that $\Phi$ is $\bD$-linear.
\end{proof}

\subsection{Periodicity phenomena for $\bD$-modules} \label{ss:periodicity-D}

Assume now that $\bk$ has characteristic $\ell>0$ and that $q$ is a root of unity in $\bk$; let $w$ be its order. We define an integer-valued sequence $b_{\bullet}$ as follows. If $w >1$, we set $b_0 = 1$, $b_i = w\ell^{i-1}$ for $i > 0$. Otherwise, we set $b_0 = 1$, $b_i = \ell^i$. It follows from \cite[\S 2]{dp} that there is an isomorphism
\[
  \bk[y_0, y_1, \ldots]/(y_i^{b_{i+1}/b_i}) \to \bD
\]
of $\bk$-algebras  given by $y_i \mapsto x^{[b_i]}$. In particular, we see that $\bD$ is a coherent ring. For a subset $X$ of $\bZ_{\ge 0}$, we let $\bD_X$ denote the subring of $\bD$ generated by $y_i$ for $i \in X$.  

Let $\bD_{\ge r}$ denote the subring of $\bD$ generated by $y_i$ for $i \ge r$.

\begin{definition}
  Let $M$ be a finitely presented graded $\bD$-module. We define
  \begin{itemize}
  \item $g_r(M)$ to be the degree of generation of $M$ as a $\bD_{\ge r}$-module.
  \item $\epsilon(M)$ to be the minimal non-negative integer $r$ so that $M$ is free as a $\bD_{\ge r}$-module.
  \item $\lambda(M)$ to be the common
    value of the expression $g_r(M) + 1 - b_r$ for $r \ge \epsilon(M) $. \qedhere
  \end{itemize}
\end{definition}

\begin{proposition}[{\cite[\S 3.3]{periodicity}}] \label{prop:lambda-epsiolon-invariants}
  We have the following:
  \begin{enumerate}[\rm \indent (a)]
  \item Let $M$ be a finitely presented $\bD$-module. Then $\dim_{\bk} M_n = \dim_{\bk} M_m$ if $n \equiv m \pmod{b_{\epsilon(M)}}$ and $n, m \ge \lambda(M)$.
  \item Let $0 \to K \to L \to M \to 0$ be a short exact sequence of finitely presented $\bD$-modules. Then we have
    \begin{align*}
      \epsilon(L) \le \max(\epsilon(K), \epsilon(M)), & \qquad	\lambda(L) = \max(\lambda(K), \lambda(M)).		
    \end{align*}
  \item Let $M$ be a finitely presented $\bD$-module and let $N$ be a finitely presented subquotient of $M$. Then $\lambda(N) \le \lambda(M)$.
  \end{enumerate}
\end{proposition}

\begin{remark} \label{rmk:itder}
Suppose $n=b_r$ for some $r$. Then ${n \brack i}_q=0$ for $0<i<n$, and so $d^n$ is a $q$-derivation on $\bD$ of degree $-n$ (Corollary~\ref{cor:itder}). Suppose that $M$ is a graded right $\bD$-module with a connection with respect to $d^n$. Then an easy variant of Proposition~\ref{prop:connection} shows that $M$ is free over $\bD_{\ge r}$. In particular, we see that $\epsilon(M) \le r$.
\end{remark}

For an integer $n$, we let $\fl(n)$ be the smallest non-negative integer $r$ such that $n \le b_r$. Note that $\fl(a+b) \le \max(\fl(a), \fl(b))+1$.

\begin{lemma}  \label{d:epsilon}
  Let $f \colon M \to N$ be a degree-preserving map of finitely presented $\bD$-modules. Let $r=\max(\epsilon(M), \epsilon(N))$. The matrix of $f$ with respect to any $\bD_{\ge r}$-module bases for $M$ and $N$  only involve the variables $y_r, \ldots, y_{s-1}$ where $s = \max(\epsilon(M), \epsilon(N), \fl(\lambda(M)))+1$.
\end{lemma}

\begin{proof}
Let $\{v_i\}$ and $\{w_j\}$ be bases for $M$ and $N$ as $\bD_{\ge r}$-modules. Then $f(v_i)=\sum_j a_{i,j} w_j$ for some $a_{i,j} \in \bD_{\ge r}$. Since the $w_j$ have non-negative degree, we find
	\begin{displaymath}
	\deg(a_{i,j}) \le \deg(v_i) < b_r+\lambda(M) \le b_s.
	\end{displaymath}
	Thus $a_{i,j}$ can only involve the variables $y_r, \ldots, y_{s-1}$.
\end{proof}

\begin{proposition}	\label{d:hom}
  Let
  \begin{displaymath}
    M_1 \stackrel{f}{\to} M_2 \stackrel{g}{\to} M_3
  \end{displaymath}
  be a complex of finitely presented $\bD$-modules, and let $H=\ker(g)/\im(f)$ be the homology. Then we have the bound
	\begin{displaymath}
	\epsilon(H) \le \max \big( \epsilon(M_1), \epsilon(M_2), \epsilon(M_3), \fl(\lambda(M_1)), \fl(\lambda(M_2)) \big)+1.
	\end{displaymath}
\end{proposition}

\begin{proof}
  Let $r=\max(\epsilon(M_1), \epsilon(M_2), \epsilon(M_3))$, and pick bases for $M_1$, $M_2$, and $M_3$ as $\bD_{\ge r}$-modules. By Lemma~\ref{d:epsilon}, we see that the matrix entries for $f$ and $g$ only involve the variables $y_r, \ldots, y_{s-1}$, where $s$ is the right side of the bound in the statement of the proposition. Let $\bD_{[r,s)}$ be the subring of $\bD$ generated by $y_r,\dots,y_{s-1}$. Let $\ol{M}_i$ be a free $\bD_{[r,s)}$-module with the same basis as $M_i$, and define $\ol{f} \colon \ol{M}_1 \to \ol{M}_2$ and $\ol{g} \colon \ol{M}_2 \to \ol{M}_3$ using the same matrices that define $f$ and $g$. Then the original sequence $M_1 \to M_2 \to M_3$ is obtained from the sequence $\ol{M}_1 \to \ol{M}_2 \to \ol{M}_3$ by applying $-\otimes_{\bk} \bD_{\ge s}$. Thus if $V=\ker(\ol{g})/\im(\ol{f})$ then $H=V \otimes_{\bk} \bD_{\ge s}$, and so $H$ is free over $\bD_{\ge s}$.
\end{proof}

\begin{proposition}	\label{d:spectral}
  Suppose $r \ge 0$. Let $\rE^{\ast,\ast}_{\ast}$ be a  cohomological spectral sequence of finitely presented $\bD$-modules supported on $\{ (a,b)  \colon -r \le a \le r,\ b \ge 0 \}$. Put
  \begin{displaymath}
    \fl^t=\max_{a+b=t}{\fl(\lambda(\rE^{a,b}_1))}, \qquad \epsilon^t_k=\max_{a+b=t}{\epsilon(\rE^{a,b}_k)}.
  \end{displaymath}
  Then for $k \ge 0$, we have the bound
  \begin{displaymath}
    \epsilon^t_{1+k} \le \max(\epsilon_1^{t-k}, \ldots, \epsilon_1^{t+k}, \fl^{t-k}, \ldots, \fl^{t+k-1})+k
  \end{displaymath}
  In particular, if there is a convergence $\rE_1^{a,b} \implies H^{a+b}$ then
  \begin{displaymath}
    \epsilon(H^t) \le \max(\epsilon_1^{t - 2r-1}, \ldots, \epsilon_1^{t + 2t+1}, \fl^{t - 2r -1}, \ldots, \fl^{t + 2r})+2r + 1.
  \end{displaymath}
\end{proposition}

\begin{proof}
  First note that $\lambda(\rE^{a,b}_{k+1}) \le \lambda(\rE^{a,b}_k)$ by Proposition~\ref{prop:lambda-epsiolon-invariants}, and so for the purposes of upper bounds, we can essentially regard $\lambda$ as unchanging from one page to the next. With this in mind, Proposition~\ref{d:hom} gives us the inequality
  \begin{displaymath}
    \epsilon_k^t \le \max(\epsilon_{k-1}^{t-1}, \epsilon_{k-1}^t, \epsilon_{k-1}^{t+1}, \fl^{t-1}, \fl^t)+1.
  \end{displaymath}
  Iterating this several times gives the bound on $\epsilon^t_{1+k}$ that appears in the proposition. For the statement about $\epsilon(H^t)$, note that $\rE^{a,b}_{\infty} = \rE^{a,b}_{2r+2}$, since on the $(2r+2)$th page and beyond, all differentials in and out of the $(a,b)$ entry are zero. Thus $H^t$ has a filtration where the associated graded pieces are $\rE^{a,b}_{2r+2}$ where $a+b=t$, and so $\epsilon(H^t) \le \epsilon^t_{2r+2}$ by Proposition~\ref{prop:lambda-epsiolon-invariants}. 
\end{proof}

\section{Finite general linear groups} \label{s:gl}

\subsection{$\VB$-modules}

Fix a prime number $p$ and a power $q$ of $p$, and assume $p$ is invertible in the field $\bk$. Let $\VB$ be the groupoid of finite dimensional $\bF_q$-vector spaces. The category $\Mod_{\VB}$ of $\VB$-modules admits a natural tensor product, due to \cite{joyal-street}, defined as follows:
\begin{displaymath}
(M \otimes N)(V) = \bigoplus_{U \subset V} M(U) \otimes N(V/U).
\end{displaymath} 
Define a map
\begin{displaymath}
\beta \colon M \otimes N \to N \otimes M
\end{displaymath}
as follows. For $x \in M(U)$ and $y \in N(V/U)$, put
\begin{displaymath}
\beta(x \otimes y) = \sum_{V=U \oplus W} \nu^{U,W}_*(y) \otimes \mu^{U,W}_*(x),
\end{displaymath}
where the sum is over all complementary spaces $W$ to $U$, and $\mu^{U,W} \colon U \to V/W$ and $\nu^{U,W} \colon V/U \to W$ are the natural isomorphisms. Then it is known that $\beta$ defines a braiding on the tensor product $\otimes$. (Joyal and Street \cite{joyal-street} proved this in the case when $\bk$ is of characteristic 0. See \cite[Theorem~14.3]{james} for the general case.)

We now consider the functor
\begin{displaymath}
\Gamma \colon \Mod_{\VB} \to \GrVec, \qquad
\Gamma(M) = \bigoplus_{n \ge 0} M(\bF_q^n)^{\bG_n}.
\end{displaymath}
We equip $\GrVec$ with its usual tensor product, but define a new braiding $\beta$ by
\begin{displaymath}
\beta(x \otimes y) = q^{\deg(x) \deg(y)} y \otimes x.
\end{displaymath}
Note that this is an isomorphism since we have assumed $q$ is invertible in $\bk$, and it is easy to verify the braid relation. We then have:

\begin{proposition}
$\Gamma$ is a strict braided monoidal functor (in a natural way).
\end{proposition}

\begin{proof}
Let $M$ and $N$ be $\VB$-modules. We then have a natural isomorphism
\begin{displaymath}
(M \otimes N)(\bF_q^n) = \bigoplus_{r+s=n} \Ind_{\bP_{r,s}}^{\bG_n}(M(\bF_q^r) \otimes N(\bF_q^s)).
\end{displaymath}
We thus see, on applying Frobenius reciprocity, that
\begin{displaymath}
(M \otimes N)(\bF_q^n)^{\bG_n} \cong \bigoplus_{r+s=n} M(\bF_q^r)^{\bG_r} \otimes N(\bF_q^s)^{\bG_s}.
\end{displaymath}
This isomorphism exactly shows that $\Gamma$ is naturally a strict monoidal functor. To be completely explicit, suppose $x \in M(\bF_q^r)^{\bG_r}$ and $y \in N(\bF_q^s)^{\bG_s}$. Then the map
\begin{displaymath}
\Gamma(M)_r \otimes \Gamma(N)_s \to \Gamma(M \otimes N)_{r+s}
\end{displaymath}
takes $x \otimes y$ to
\begin{displaymath}
\sum_{V \subset \bF_q^n} \rho^V_*(x) \otimes \rho^{\bF^n/V}_*(y),
\end{displaymath}
where the sum is over all $r$ dimensional subspaces $V$ of $\bF_q^n$, and for a $t$ dimensional vector space $U$ we let $\rho^U \colon \bF_q^t \to U$ be an arbitrary isomorphism. Note that $\rho^V_*(x)$ is independent of the choice of isomorphism since $x$ is $\bG_r$ invariant.

We now show that $\Gamma$ is compatible with the braiding. For this, we must verify that the diagram
\begin{displaymath}
\xymatrix{
\Gamma(M)_r \otimes \Gamma(N)_s \ar[r] \ar[d]_{\beta} & \Gamma(M \otimes N)_{r+s} \ar[d]^{\Gamma(\beta)} \\
\Gamma(N)_s \otimes \Gamma(M)_r \ar[r] & \Gamma(N \otimes M)_{r+s} }
\end{displaymath}
commutes. Thus let $x \in \Gamma(M)_r$ and $y \in \Gamma(N)_s$ be given. Going to the right and then down, we obtain
\begin{displaymath}
\sum_{V \subset \bF_q^n} \sum_{\bF_q^n=V \oplus U} \nu^{V,U}_*(\rho^{\bF_q^n/V}_*(y)) \otimes \mu^{V,U}_*(\rho^V_*(x)),
\end{displaymath}
where the outer sum is over $r$-dimensional subspaces $V$ of $\bF_q^n$, the inner sum is over complements $U$ to $V$, and $\mu^{V,U} \colon V \to \bF_q^n/U$ and $\nu^{V,U} \colon \bF_q^n/V \to U$ are the natural isomorphisms. Now, $\nu^{V,U} \circ \rho^{\bF_q^n/V}$ is an isomorphism $\bF_q^s \to U$, and so we may as well call it $\rho^U$, and similarly we may as well replace $\mu^{V,U} \circ \rho^V$ with $\rho^{\bF^n/U}$. We thus obtain
\begin{displaymath}
\sum_{\bF_q^n=V \oplus U} \rho^U_*(y) \otimes \rho^{\bF_q^n/U}_*(x).
\end{displaymath}
Now, the choice of $V$ in the decomposition does not affect the value of the summand. For each $U$, there are $q^{rs}$ choices of complement $V$, and so the above sum is equal to
\begin{displaymath}
q^{rs} \sum_{U \subset \bF_q^n} \rho^U_*(y) \otimes \rho^{\bF_q^n/U}_*(x),
\end{displaymath}
where the sum is over $s$-dimensional subspaces $U$ of $\bF_q^n$.

Now suppose we first go down in the diagram and then to the right. Going down yields $q^{rs} y \otimes x$, and then going to the right gives
\begin{displaymath}
q^{rs} \sum_{U \subset \bF_q^n} \rho^U_*(y) \otimes \rho^{\bF_q^n/U}_*(x),
\end{displaymath}
where the sum is over $s$-dimensional subspaces $U$ of $\bF_q^n$. We have thus verified that the diagram commutes, which completes the proof.
\end{proof}

We now establish the following somewhat technical result that we will need when considering $\rR \Gamma$:

\begin{proposition}
Let $I$ and $J$ be injective $\VB$-modules. Then $I \otimes J$ is $\Gamma$-acyclic.
\end{proposition}

\begin{proof}
It suffices to treat the case where $I$ and $J$ are each concentrated in a single degree, say degrees $r$ and $s$. Thus we treat $I$ as an injective $\bk[\GL_r]$-module and $J$ as an injective $\bk[\GL_s]$-module. We can then identify $I \otimes J$ with the $\bk[\GL_{r+s}]$-module $\Ind_{\bP_{r,s}}^{\GL_{r+s}}(I \otimes J)$. We have
\begin{displaymath}
\rR^i \Gamma(I \otimes J) = \rH^i(\GL_{r+s}, \Ind_{\bP_{r,s}}^{\GL_{r+s}}(I \otimes J)) = \rH^i(\bP_{r,s}, I \otimes J)
\end{displaymath}
by Shapiro's lemma. Now, let $\bU_{r,s}$ be the unipotent subgroup of $\bP_{r,s}$, which is normal with quotient $\GL_r \times \GL_s$. We thus have a spectral sequence
\begin{displaymath}
\rE^{a,b}_2 = \rH^a(\bU_{r,s}, \rH^b(\GL_r \times \GL_s, I \otimes J)) \implies \rH^{a+b}(\bP_{r,s}, I \otimes J).
\end{displaymath}
However, $\bU_{r,s}$ is a $p$-group and since $p$ is invertible in $\bk$ it has no higher cohomology. We conclude
\begin{displaymath}
\rH^i(\bP_{r,s}, I \otimes J)=\rH^i(\GL_r \times \GL_s, I \otimes J)^{\bU_{r,s}},
\end{displaymath}
which vanishes for $i>0$ by the K\"unneth formula.
\end{proof}

\subsection{First proof of Theorem~\ref{thm:q-construction}}

Let $\bA$ be the $\VB$-module defined by $\bA(V)=\bk \cdot t^V$. This is naturally an algebra in $\Mod_{\VB}$. Let $\bk\{1\}$ be the $\VB$-module that is $\bigoplus_{v \in L \setminus \{0\}} \bk \cdot e_v$ on one-dimensional spaces $L$, and~0 on all other spaces. We note that $\bk\{1\}$ is $\Gamma$-acyclic, as $(\bk\{1\})(L)$ is the regular representation of $\GL(L)$ when $L$ is one-dimensional.

Define a map
\begin{displaymath}
d \colon \bA \to \bA \otimes \bk\{1\}, \qquad
d(t^V) = \sum_{C \subset V} \sum_{v \in (V/C) \setminus \{0\} } t^{C} \otimes e_v,
\end{displaymath}
where the sum is over subspaces $C$ of $V$ of codimension~1. As in the symmetric group case, we have:

\begin{proposition}
The map $d$ is a derivation.
\end{proposition}

\begin{proof} 
Consider a short exact sequence
\begin{displaymath}
0 \to U \to V \to W \to 0.
\end{displaymath}
It suffices to prove
\begin{displaymath}
d(t^V) = t^U \cdot d(t^W) + d(t^U) \cdot t^W,
\end{displaymath}
where $d(t^U) \cdot t^W$ indicates to first apply the braiding to switch $t^W$ and the $e$'s, and then multiply. More precisely, $d(t^U) \cdot t^W = (m \otimes 1)(1 \otimes \beta)(d(t^U) \otimes t^W)$. We have
\begin{displaymath}
d(t^V) = \sum_{C \subset V} \sum_{v \in (V/C) \setminus \{0\}} t^C \otimes e_v 
\end{displaymath}
and \begin{align*}
t^U \cdot d(t^W) &= t^U \cdot \sum_{C' \subset W} \sum_{v \in (W/ C') \setminus \{0\}} t^{C'}\otimes e_v \\
& = \sum_{U \subset C \subset V} \sum_{v \in (V/ C) \setminus \{0\}} t^{C}\otimes e_v. 
\end{align*}
Now, we have
\begin{align*}
d(t^U) \cdot t^W &= (m \otimes 1)(1 \otimes \beta)\left(\sum_{C' \subset U} \sum_{v \in  (U/C') \setminus \{0\}} t^{C'} \otimes e_v \otimes t^{W}\right) \\
&= (m \otimes 1)\left(\sum_{C' \subset U} \sum_{V/C' = U/C' \oplus C''} \sum_{v \in  (U/C') \setminus \{0\}} t^{C'} \otimes t^{C''} \otimes e_v\right) \\
& = \sum_{C \subset V, U \not\subset C} \sum_{v \in  (V/C) \setminus \{0\}} t^{C}  \otimes e_v,
\end{align*} where the last equality above follows from the following bijective correspondence: \begin{align*}
\{(C', C'') \colon C' \subset U, V/C' = U/C' \oplus C'', \dim C' = \dim U -1   \} &\to \{C \colon U \not\subset C, \dim C = \dim V -1 \}\\
(C', C'') &\mapsto (V \to V/C')^{-1}(C'')\\
(C \cap U, C/(C \cap U)) & \mapsfrom C, 
\end{align*}  the multiplication $m(t^{C'} \otimes t^{C''}) = t^C$, and the fact that $U/C' = V/C$. Now it is clear that \begin{displaymath}
d(t^V) = t^U \cdot d(t^W) + d(t^U) \cdot t^W,
\end{displaymath} completing the proof.
\end{proof}

\begin{proposition}  $\Gamma(\bA)$ is the $q$-divided power algebra $\bD$, and the derivation on it takes $x^{[n]}$ to $x^{[n-1]}$.
\end{proposition}

\begin{proof}
  Note that $\Gamma(\bA)$ is isomorphic to $\bD$ as $\bk$-vector spaces. It suffices to check that the multiplication on  $\Gamma(\bA)$ agrees with the one on $\bD$. The multiplication map \[\Gamma(\bA)_n \otimes \Gamma(\bA)_m \to \Gamma(\bA_n \otimes \bA_m) \to \Gamma(\bA)_{n+m}\] is given by the composite
  \[
    \bk \otimes \bk = \Gamma(\bA)_n \otimes \Gamma(\bA)_m \to \rH^0(\bP_{n,m}) \to \rH^0(\bG_{n+m} ,\Ind_{\bP_{n,m}}^{\bG_{n+m}} \bk) \to \rH^0(\bG_{n+m}) = \Gamma(\bA)_{n+m} = \bk,
  \]
  which takes
  \[
    a \otimes b  \mapsto ab \mapsto \sum_{\gamma \in \bG_{n,m}/P_{n,m}} \gamma ab \to {n +m \brack n}_q ab.
  \]
  Thus $\Gamma(\bA)$ is the $q$-divided power algebra. 

Now note that $\Gamma(\bk\{1\}) = \bk\{1\}$. The derivation
\[
  \Gamma(\bA) \to \Gamma(\bA) \otimes\bk\{1\} = \Gamma(\bA)[1]
\]
in degree $n$ is given by the composite
\[
  \bA_n^{\bG_n} \to (\bA_{n-1}\otimes \bk\{1\}  )^{\bG_n} \to \bk^{\bP_{n-1, \ol{1}}} \to \bA_{n-1}^{\bG_{n-1}},
\]
where the second map comes from Shapiro's lemma, and the third map is the inverse of the inflation isomorphism. It is clear that, under the isomorphism $\Gamma(\bA) = \bD$, this map  takes $x^{[n]}$ to $x^{[n-1]}$. This completes the proof.
\end{proof}

We are now in a position to apply the general theory of \S \ref{s:congen} to prove Theorem~\ref{thm:q-construction}. We find that $\rR \Gamma(\bA)$ is an algebra with derivation of degree $-1$; moreover, $\Gamma(\bA)$ is a subalgebra with derivation and each $\rR^i \Gamma(\bA)$ is a $\Gamma(\bA)$-module with connection. Note here that ``derivation'' and ``connection'' are taken with respect to the braiding $\beta$; concretely, these translate to $q$-derivation and $q$-connection. By the previous proposition, $\Gamma(\bA)$ is the $q$-divided power algebra, and the derivation on it takes $x^{[n]}$ to $x^{[n-1]}$. Moreover, the connection on $\rR^i \Gamma(\bA)=\bigoplus_{n \ge 0} \rH^i(\bG_n, \bk)$ is induced from the derivation $d \colon \bA \to \bA \otimes \bk\{1\}$, and thus is the map
\begin{align*}
\rH^i(\bG_n, \bk) &\to \rH^i(\bG_n, (\bA \otimes \bk\{1\} )(\bF_q^n)) \cong \rH^i(\bG_n, \Ind_{\bP_{n-1,\ol{1}}}^{\GL_n}(\bk)) \\
&\cong \rH^i(\bP_{n-1,\ol{1}}, \bk) \cong \rH^i(\bG_{n-1},\bk),
\end{align*}
which is just the restriction map (corresponding to the subgroup $\bP_{n-1,\ol{1}}$) composed with the inverse of the inflation map (corresponding to the projection $\bP_{n-1,\ol{1}} \to \bG_{n-1}$). 

\subsection{A direct proof of the main theorem}

We now prove Theorem~\ref{thm:q-construction} directly, without any category theory. See \cite[\S 4.2]{periodicity} for the symmetric group analog of this argument. By $\bP_{n_1, n_2, \ldots,n_d}$, we mean the group of $n_d\times n_d$ block upper-triangular invertible matrices over $\bF_q$ with block sizes $n_1, n_2, \ldots , n_d$. An overline $\ol{n}_i$ signifies that the $i$th diagonal block must be the identity matrix, and a superscript $\bP^{(i,j)}$ signifies that the $(i,j)$ block must be 0.

\begin{lemma}	\label{lem:mid-portion}
  The following diagram  has the property that the middle path is the sum of the top and the bottom paths.
  \begin{displaymath}
    \begin{tikzcd}
      \rH^{t}(\bP_{n, m-1, \ol{1}}) \ar[bend left=15]{rrd}{\cor} \\
      \rH^{t}(\bP_{n,m} ) \ar{r}{\cor}  \ar[swap]{u}{\res}  \ar{d}{\res} & \rH^{t}(\bG_{n+m} ) \ar{r}{\res} & \rH^{t}(\bP_{n + m-1,\ol{1}} )  \\
      \rH^{t}(\bP^{(2,3)}_{n-1, \ol{1},  m} )  \ar{rr}{\zeta^{*}} & & \rH^{t}(\bP^{(2,3)}_{n-1, m,  \ol{1}} ) \ar{u}{\cor}
    \end{tikzcd}
  \end{displaymath}
  Here $\zeta$ in the permutation that switches blocks $2$ and $3$, and $\cor$ denote the corestriction (i.e., the transfer map).
\end{lemma}

\begin{proof}
  We claim that $|\bP_{n + m-1,\ol{1}} \backslash \bG_{n+m}/ \bP_{n,m}| = 2$. To see this, first note that $\bG_{n+m}/ \bP_{n,m}$ can be identified with the set of $n$-dimensional subspaces in $\bF_q^{n+m}$ via sending a matrix $g$ to the span of its first $n$ columns. Since $\bP_{n + m-1,\ol{1}}$ contains $\bG_{n+m-1}$, we see that all $n$-dimensional subspaces contained in $\bF_q^{n+m-1}$ form a single orbit under the action of $\bP_{n + m-1,\ol{1}}$. Now suppose $W$ is a subspace not contained in $\bF_q^{n+m-1}$. It is easy to see that there is an element $\sigma$ of $\bP_{n + m-1,\ol{1}}$, such that $\be_{n+m}  \in \sigma W$. But then $W$ can be written as $W_1 \oplus \bF_q \be_{n+m}$ where $W_1$ is an $(n-1)$-dimensional subspace of $\bF_q^{n+m-1}$. By transitivity of the action of $\bG_{n+m-1}$, we can assume that $W_1$ is the span of $\be_1, \be_2, \ldots, \be_{n-1}$. This proves the claim.
	
  Now note that $\bP_{n,m}$ is the stabilizer of the space spanned by $\be_1, \be_2, \ldots, \be_{n}$. Let $\bH$ be the stabilizer of the space spanned by $\be_1, \be_2, \ldots, \be_{n-1}, \be_{n+m}$. We see that
  \[
    \bP_{n + m-1,\ol{1}} \cap \bP_{n,m} = \bP_{n, m-1, \ol{1}}, \qquad
    \bP_{n + m-1,\ol{1}} \cap \bH = \bP^{(2,3)}_{n-1, m,  \ol{1}}.
  \]
  The rest follows from \cite[Proposition~III.9.5]{bro}.
\end{proof}

\begin{lemma} \label{lem:coeffaceability}
  Suppose $G_1 \subset G_2$ are finite groups. Let  $N_2$ be a normal subgroup of  $G_2$, and set $N_1 = N_2 \cap G_1$. Let $V$ be a $\bk[G_2]$-module. Assume that the following conditions hold:
  \begin{enumerate}
  \item There is a subgroup $G \subset G_2$ containing $N_2$ and $G_1$ such that the inclusion $i \colon G_1/N_1 \to G/N_2$ is an isomorphism.
    
  \item The size of $N_2$ is invertible in $\bk$.
		
  \item $N_2$ acts trivially on $V$.
  \end{enumerate}
  Let $V$ be a $\bk[G_2]$-module. Then in the diagram below, the top path is $|N_2/N_1|$-times the bottom path. 	
  \begin{displaymath}
    \begin{tikzcd}	
      \rH^t(G_1, V) \ar{d}{\cong} \ar{r}{\cor}	&  \rH^t(G_2, V)   \\
      \rH^t(G_1/N_1, V) \ar{r}{\cor}	&  \rH^t(G_2/N_2, V) \ar{u}{\cong}
    \end{tikzcd}
  \end{displaymath}
Here the right vertical map is the inflation map and the left vertical map is the inverse of the inflation map.
\end{lemma}

\begin{proof}
  Fix a $\bk[G_2/N_2]$-injective resolution $0 \to V \to I^{\bullet}$. We let $G_2$ act on $I$ via the projection $G_2 \to G_2/N_2$. It is also clear that $I^{i}$ is injective as a $K/(K \cap N_2)$-module for any subgroup $K$ of $G_2$. By Hypothesis (b),  $I^i$ is acyclic with respect to the functor $\rH^0(K, -)$ where $K$ is any subgroup of $G_2$. Hence we can calculate all the cohomology groups in the diagram above using the complex $I^{\bullet}$. Thus it suffices to prove the assertion in the case $t=0$, which follows from the easy identity
  \[
    [G_2:G_1] = |N_2/N_1| \cdot [G_2/N_2 : G_1/N_1]. \qedhere
  \]
\end{proof}

\begin{proof}[Proof of Theorem~\ref{thm:q-construction}]
  Note that the second assertion, that $d$ is surjective, follows from the first assertion and Proposition~\ref{prop:connection}. So it suffices to prove the first assertion. For that, we consider the following diagram.
	
	 {\tiny
		\begin{displaymath}
		\begin{tikzcd}
			\rH^a(\bG_n) \otimes \rH^b(\bG_{m-1}) \ar{r}{\res \circ \kappa}	& \rH^t(\bP_{n,m-1}) \ar{d}{\cong} \ar[bend left =18]{rrrdd}{\cor} \\
	\rH^a(\bG_n) \otimes 	\rH^b(\bP_{m-1, \ol{1}}) \ar{r}{\res \circ \kappa}   \ar{u}{\id \otimes \cong}  & \rH^{t}(\bP_{n, m-1, \ol{1}}) \ar[bend left=15]{rrd}{\cor} \\
		{  \rH^a(\bG_n ) \otimes 	\rH^b(\bG_{m})}  \ar{u}{\id \otimes \res  } \ar[swap]{d}{\res \otimes \id  } \ar{r}{\res \circ \kappa} & \rH^{t}(\bP_{n,m} ) \ar{r}{\cor}  \ar[swap]{u}{\res}  \ar{d}{\res} & \rH^{t}(\bG_{n+m} ) \ar{r}{\res} & \rH^{t}(\bP_{n + m-1,\ol{1}} ) \ar{r}{\cong} &  \rH^{t}(\bG_{n + m-1} ) \\
		 \rH^a(\bP_{n-1,\ol{1}} ) \otimes \rH^b(\bG_{m}) \ar{r}{\res \circ \kappa} \ar[swap]{d}{\cong \otimes \id} & \rH^{t}(\bP^{(2,3)}_{n-1, \ol{1},  m} ) \ar{d}{\cong} \ar{rr}{\zeta^*} & & \rH^{t}(\bP^{(2,3)}_{n-1, m,  \ol{1}} ) \ar{u}{\cor} \\
	\rH^a(\bG_{n-1} )  \otimes \rH^b(\bG_{m}) \ar{r}{\res \circ \kappa} & \rH^{t}(\bP_{n-1,  m} ) \ar[bend right =24, swap]{rrruu}{\cor}
		\end{tikzcd}
		\end{displaymath}
	} Here all the isomorphisms $\cong$ are either the inflation isomorphisms or their inverses, $\kappa$ denotes K\"unneth map, and $t = a+b$. Let $x \in \rH^a(\bG_n)$ and $y \in \rH^b(\bG_m)$. By construction, the top, the middle and the bottom paths applied to $x\otimes y$ are just $x d(y)$, $d(xy)$, and  $d(x) y$ respectively. Thus it suffices to show that
  \[
    \text{middle = top + $q^m$ (bottom)}.
  \]
  This follows immediately from  Lemma~\ref{lem:mid-portion} and the following claim:

\vskip.5\baselineskip\noindent
{\bf Claim}. All the regions on the periphery of the diagram are commutative except the bottom right, where the top path is $q^m$ times the bottom path:
    \begin{displaymath}
      \begin{tikzcd}	
        & & \rH^{t}(\bP_{n + m-1,\ol{1}} ) \ar{r}{\cong} &  \rH^{t}(\bG_{n + m-1} ) \\
        \rH^{t}(\bP^{(2,3)}_{n-1, \ol{1},  m} ) \ar{d}{\cong} \ar{rr}{\zeta^{*}} & & \rH^{t}(\bP^{(2,3)}_{n-1, m,  \ol{1}} ) \ar{u}{\cor} \\
        \rH^{t}(\bP_{n-1,  m} ) \ar[bend right =24, swap]{rrruu}{\cor}
      \end{tikzcd}
    \end{displaymath}

{\bf Proof of Claim.} The assertion on the bottom right region follows from  Lemma~\ref{lem:coeffaceability} by setting $G_1 = \bP^{(2,3)}_{n-1, m,  \ol{1}}$, $G_2 = \bP_{n + m-1,\ol{1}}$, $G = \bP_{n-1, m,  \ol{1}}$, and $N_2$ is the unipotent radical of $G_2$. We note here that the composite \[\rH^{t}(\bP^{(2,3)}_{n-1, m,  \ol{1}} ) \xrightarrow{(\zeta^{\ast})^{-1}}  \rH^{t}(\bP^{(2,3)}_{n-1, \ol{1},  m} ) \xrightarrow{\cong} \rH^{t}(\bP_{n-1,  m} ) \] is the inflation map for the projection $\bP^{(2,3)}_{n-1, m,  \ol{1}} \to \bP_{n-1,  m}$, and so Lemma~\ref{lem:coeffaceability} applies.

Similarly, the commutativity of the top right part follows from  Lemma~\ref{lem:coeffaceability} by setting $G_1 = G = \bP_{n, m-1,  \ol{1}}$, and $G_2 = \bP_{n + m-1,\ol{1}}$. The commutativity of the other four diagrams can be checked easily when $t=0$, and a similar argument as in the first paragraph of the proof of Lemma~\ref{lem:coeffaceability} can be used to see that $t=0$ case suffices. \qedhere
\end{proof}

\subsection{Proof of Corollary~\ref{cor:free-D-module}}

Note that $\bE^t = \bigoplus_{n \ge 0} \rH^t(\bG_n)$ is a module over $\bE^0 = \bD$  under the transfer product. Let $d \colon \bE^0 \to \bE^0[1]$, and $\nabla \colon \bE^t \to \bE^t[1]$ be the restrictions of the derivation $d \colon \bE \to \bE[1]$ as in Theorem~\ref{thm:q-construction}.  Upon identifying $\bE^0$ with $\bD$, the derivation $d$ is given by $x^{[n]} \to x^{[n-1]}$ on $\bD$, and $\nabla$ is  a $q$-connection on the $\bD$-module $\bE^t$ (Theorem~\ref{thm:q-construction}). Thus by Proposition~\ref{prop:connection}, we see that $\bE^t $ is isomorphic to $\bD \otimes_{\bk} \ker(\nabla)$. In particular, $\bE^t $ is a free $\bD$-module. To complete the proof, it suffices to show that $\ker(\nabla)$ is supported in degrees $\le t$, or that $\dim_{\bk} \rH^t(\bG_n)$ is constant for $n \ge 2t$. This follows from the homological stability result of Maazen \cite{maazen} (see also \cite[Theorem 4.11]{vanderkallen}). In the case when $q \neq 2$, the bound can be improved to $n \ge t$  using \cite[Theorem~A]{sprehn-wahl}. The last assertion follows from this.

\section{Periodicity in the cohomology of $\VI$-modules}
\label{s:vi}

\subsection{Application to $\VI$-modules}

Recall that $\bA$ is the algebra in $\Mod_{\VB}$ that is trivial in every degree. The category of right $\bA$-modules is equivalent to the category of $\VI$-modules (see \cite[Proposition~2.1]{VI-shift}), and we will not distinguish between the two. For a $\VB$-module $W$, we let $\cI(W) = W\otimes  \bA$; we refer to such $\VI$ modules as {\bf induced}.  In the case when $q$ is invertible in $\bk$, the category $\Mod_{\VI}$ is well understood. In particular, we have the following:

\begin{theorem}[{\cite[Theorem 1.1]{VI-shift}}]
Let $M$ be a finitely generated $\VI$-module. Then there exists a polynomial $P$  such that $\dim_{\bk} M(\bF_q^n) = P(q^n)$ for $n \gg 0$.
\end{theorem}

We denote, by $\delta(M)$, the degree of the polynomial $P$ as in the theorem above. We now state our application to $\VI$-modules (recall the notation $\fl$ introduced in \S\ref{ss:periodicity-D}):

\begin{theorem}	\label{thm:VImod}
  Assume that $p \neq \ell$. Let $M$ be a finitely generated $\VI$-module generated in degrees $\le t_0$ and related in degrees $\le t_1$. For each $t \ge 0$, there is a finitely presented $\bD$-module $K$ and a $\bD$-module map $i \colon \rR^t \Gamma(M) \to K$ such that the following holds:
  \begin{enumerate}[\rm \indent (a)]
  \item The map $i$ is an isomorphism in degrees $\ge t_0 + t_1$.
  \item $\lambda(K) \le \begin{cases}
      t + \delta(M) \qquad &\mbox{if } q \neq 2\\
      2t + \delta(M) \qquad &\mbox{if } q = 2
    \end{cases}$.
  \item $\epsilon(K) \le \begin{cases}
  \fl(t + 4\delta(M)) + 2 \delta(M) + 1 \qquad &\mbox{if } q \neq 2\\
  \fl(2t + 7\delta(M)) + 2 \delta(M) + 1 \qquad &\mbox{if } q = 2
  \end{cases}$.
  \end{enumerate}
  In particular, $\dim_{\bk} \rH^t(\bG_n, M(\bF_q^n))$ is periodic for $n \ge \max(\lambda(K), t_0 + t_1)$ with period $b_{\epsilon(K)}$.
\end{theorem}

\begin{remark}
	The proof below is based on  \cite[Theorem~4.17]{periodicity} for $\FI$-modules, whose proof contains an unjustified statement. Our proof below can be adapted to give an alternate proof of \cite[Theorem~4.17]{periodicity} but with possibly worse bound on $\epsilon$. 
\end{remark}

\begin{proof}
  By \cite[\S 5.2]{VI-shift}, we have an exact triangle of the form \[ T \to M \to F \to \] such that $T$ is a finite complex of $\VI$-modules supported in finitely many degrees and $F$ is a complex
  \[
    \cdots \to F^{-1} \to F^0 \to F^1 \to \cdots
  \]
  of finitely generated induced modules each generated  in degrees $\le \delta(M)$ such that $F^p = 0$ for $\abs{p} > \delta(M)$. By \cite{VI-homology}, the cohomologies of $T$ are supported in degrees $< t_0 + t_1$. Thus $\rR^t\Gamma(T)$ is supported in degrees $< t_0 + t_1$ as well for each $t$. Set $K = \rR^t \Gamma(F)$. We conclude that the natural map $i \colon \rR^t \Gamma(M) \to K$, induced by the triangle above, is an isomorphism in degrees $\ge t_0 + t_1$. 
	
  Denote, by $\rR^{\bullet} \Gamma$, the functor $ \bigoplus_{t \ge 0} \rR^t \Gamma$. Then $\rR^{\bullet} \Gamma$ is easily seen to be strong monoidal. Thus, for an induced module $\cI(W)$, we have
  \[
    \rR^{\bullet} \Gamma (\cI(W)) = \rR^{\bullet} \Gamma(W \otimes \bA) = \rR^{\bullet}\Gamma (W) \otimes_{\bk}\rR^{\bullet} \Gamma(\bA)  =   \rR^{\bullet}\Gamma (W) \otimes_{\bk} \bE^{\bullet} .
  \]
  In particular, if $W$ is supported in degrees $\le \delta(M)$, then by Corollary~\ref{cor:free-D-module} we see that
  \[
    \lambda(\rR^{b} \Gamma (\cI(W))) \le \begin{cases}
      b + \delta(M) \qquad &\mbox{if } q \neq 2\\
      2b + \delta(M) \qquad &\mbox{if } q = 2
    \end{cases}.
  \]
  Part (b) now follows immediately from the previous paragraph and Proposition~\ref{prop:lambda-epsiolon-invariants}(b).
	
  We have a spectral sequence:
  \[
    \rE^{a,b}_1 = \rR^b \Gamma (F^a) \implies \rR^{a+b}\Gamma(F).
  \]
  We now appeal to Proposition~\ref{d:spectral}. In the notation of that proposition, we have
  \[
    \fl^t = \max_{a+b=t} \fl(\lambda(\rR^b\Gamma(F^a))) \le \begin{cases}
      \fl(t + 2\delta(M)) \qquad &\mbox{if } q \neq 2\\
      \fl(2t + 3\delta(M)) \qquad &\mbox{if } q = 2
    \end{cases},
  \]
  and $\epsilon^1_t = 0$ (by Corollary~\ref{cor:free-D-module}). This immediately implies Part (c).
\end{proof}

\subsection{Periodicity in the cohomology of unipotent Specht modules}

We apply the theorem above to unipotent Specht modules. As in James \cite{james}, by a composition of $d$ we mean a sequence $\mu = (\mu_1, \mu_2, \ldots)$ of non-negative numbers, whose sum $|\mu|$ is $d$. We let $\bM_{\mu}$  to be the unipotent Specht module  corresponding to a composition $\mu = (\mu_1, \mu_2, \ldots)$ (see \cite[Definition 11.11]{james} where it is denoted $S_\lambda$). We denote the permutation module of $\bG_d$ on the corresponding parabolic subgroup by $\bP_{\mu}$. This is the permutation representation on flags of the form
\[
  \bF_q^d =V_0 \supseteq V_1 \supseteq V_2 \supseteq \cdots \supseteq V_{\ell(\mu)} = 0
\]
such that $\dim(V_{i-1}/V_i)=\mu_i$. Note that $\bM_{\mu} \subset \bP_{\mu}$, and  if $\lambda = \lambda_1 \ge \lambda_2 \ge \cdots$ is the partition of $d$ obtained by rearranging the parts $\mu$, then $\bM_{\mu}\cong \bM_{\lambda}$ (see \cite[Theoem~16.1]{james}). 

For $n \ge |\mu|$, let $\mu[n]$ denote the composition $(n - \abs{\mu}, \mu)$. Given integers $1 \le r \le \ell(\mu)$ and $0 \le i \le \mu_r$, we also set
\[
  \mu^{r,i} = (\mu_1,\dots,\mu_{r-1}, \mu_r + \mu_{r+1}-i, i, \mu_{r+2}, \dots, \mu_{\ell(\mu)})
\]
and define a map $\psi_{r,i} \colon \bP_\mu \to \bP_{\mu^{r,i}}$ which sends the flag $V_\bullet$ to the sum of all flags where $V_r$ is replaced by any subspace $W_r$ such that $\dim(W_r / W_{r+1}) = i$. We define a set
\[
  K(\mu) = \{ (r,i) \mid 2 \le r \le \ell(\mu), \quad 0 \le i \le \mu_r-1\}.
\]
By \cite[Theorem 15.18]{james}, if $\mu$ is a partition, we have
\[
  \bM_\mu = \bigcap_{(r,i) \in K(\mu)} \ker \psi_{r-1,i}.
\]
Note that $K(\mu[e]) = K(\mu[e'])$ for any $e,e' \ge |\mu|$ and that for $(r,i) \in K(\mu[e])$, if $e' \ge e$, we have $\mu[e']^{r,i} = \mu[e]^{r,i} + (e'-e)$ where $+$ denotes adding $e'-e$ to the first part of $\mu[e]^{r,i}$.

\begin{proposition}
  Let $\mu$ be a partition of $d$. Then there exists a $\VI$-module $\bL_{\mu}$ generated in degrees $\le 2d + 1$, related in degrees $\le 2d + 2$ such that $\bL_{\mu}(\bF_q^n) \cong \bM_{\mu[n]}$ for each $n \ge d+\mu_1$.
\end{proposition}

\begin{proof}
  First note that $\cI(\bP_{\mu}) = \bigoplus_{n \ge d} \bP_{\mu[n]}$. Define a map
  \[
    \cI(\bP_\mu) \to \bigoplus_{(r,i) \in K(\mu[d])} \cI(\bP_{\mu^{r,i}})
  \]
  using the maps $\psi_{r,i}$. From the explicit description of $\psi_{r,i}$, the induced maps in each degree are also the maps $\psi_{r,i}$. Denote the kernel and the cokernel by $\bL_{\mu}$ and $C$, respectively. By the above discussion, we have $\bL_{\mu}(\bF_q^n) \cong \bM_{\mu[n]}$ whenever $n \ge d+\mu_1$ (i.e., when $\mu[n]$ is a partition). The exact sequence
  \[
    0 \to \bL_{\mu} \to  \cI(\bP_{\mu})  \to \bigoplus_{(r,i) \in K(\mu[d])} \cI(\bP_{\mu^{r,i}}) \to C \to 0
  \]
  and a spectral sequence argument shows that
  \[
    \deg \Tor_i^{\bA}(\bk, \bL_{\mu}) \le \max(d, \deg \Tor_{i+2}^{\bA}(\bk, C)).
  \]
  By \cite[Theorem~1.1(b)]{VI-homology}, we conclude that $\reg(C) \le 2d-1$. Thus $\bL_{\mu}$ is generated in degrees $\le 2d+1$, and related in degrees $\le 2d+2$.
\end{proof}

\begin{theorem} 
Assume that $\ell \neq p$.  Let $\mu$ be a composition of $d$, and let $s \ge 0$ be the smallest integer such that \[ s \ge \begin{cases}
\log_{\ell}(2t + 7d) &\mbox{if } q = 2\\
\log_{\ell}(t + 4d) &\mbox{if } q \neq 2.
\end{cases}  \] Then $n \mapsto \dim_{\bk} \rH^t(\bG_{n}, \bM_{\mu[n]})$ is periodic in the range $n \ge \max(d + 2t, 4d + 3)$ with period $w \ell^{s + 2d+1}$.
\end{theorem}

\begin{proof}
	Let $\bL_{\mu}$ be the $\VI$-module as in the previous proposition. By the hook-length formula, it is easy to check that the dimension of $\bL_{\mu}(\bF_q^n)$ is eventually a polynomial in $q^n$ of degree $d$. The assertion now follows immediately from Theorem~\ref{thm:VImod}.
\end{proof}

\begin{remark}
In \cite[Theorem 7.2.4]{harman-thesis} Harman has proven equivalences for the categories of unipotent representations for finite general linear groups which imply many periodicity statements about the structure of modules. As far as we can tell, this does not imply our result on cohomology nor does it follow from it.
\end{remark}

\end{document}